\documentclass[11pt,english]{smfart}
\usepackage[latin1]{inputenc}
\usepackage[T1]{fontenc}
\usepackage[english,french]{}
\usepackage{hyperref}
\usepackage{amssymb,amsmath,url,xspace,smfthm,euscript,enumerate}
\usepackage{fancyhdr}
\usepackage{layout}
\newtheorem{theorem}{Theorem}[section]
\newtheorem{lemma}{Lemma}[section]

\newtheorem{corollary}{Corollary}[section]
\newtheorem{definition}{Definition}[section]
\newtheorem{remark}{Remark}[section]
\newtheorem{remarks}{Remarks}[section]

\newtheorem{heuristic}{Heuristic}[section]
\newtheorem{conjecture}{Conjecture}[section]
\numberwithin{equation}{section}
  \def\N{\mathbb{N}}

  \def\Q{\mathbb{Q}}
  \def\Z{\mathbb{Z}}
  \def\F{\mathbb{F}}
 \def\zz{\mathbb{Z}}
\def\virg{\raise 2pt \hbox{,}\,\,}

\def\plus{\ds\mathop{\raise 2.0pt \hbox{$\bigoplus $}}\limits}
\def\mult{\ds\mathop{\raise 2.0pt \hbox{$\bigotimes$}}\limits}
\def\prd{ \ds\mathop{\raise 2.0pt \hbox{$  \prod   $}}\limits}
\def\Cap{ \ds\mathop{\raise 2.0pt \hbox{$\bigcap   $}}\limits}
\def\Cup{ \ds\mathop{\raise 2.0pt \hbox{$\bigcup   $}}\limits}
\def\sm{  \ds\mathop{\raise 2.0pt \hbox{$  \sum    $}}\limits}

\let\ds=\displaystyle

\def \tensorZp{\otimes{\raise -0.8pt \hbox{\!\!$_{_{\zz_{\!p}}}$}}}
\def \tensorZ{\otimes{\raise -0.8pt \hbox{\!\!$_{_{\zz}}$}}}

\theoremstyle{plain}

\begin{document}

\title[On the order modulo $p$ of an algebraic number]{On the order modulo $p$ \\ of an algebraic number}
 
\author[Georges {\sc Gras}]{{\sc Georges} GRAS}

\address{Georges {\sc Gras}, 
Villa la Gardette, chemin Ch\^ateau Gagni\`ere, 38520 Le Bourg d'Oisans} 
\email{g.mn.gras@wanadoo.fr} 
\urladdr{\url{http://www.researchgate.net/profile/Georges_Gras}}


\thanks{Some very fast programs, written in this paper, were suggested by Bill Allombert
to obtain the divisors of $p^2-1$ and Euler function computations when $p$ is very large.
I thank him very warmly for his advices and many improvements of my first ``basic'' programs.}

\begin{abstract}
Let $K/\Q$ be Galois, and let $\eta\in K^\times$ 
whose conjugates are multiplicatively independent. 
For a prime $p$, unramified, prime to $\eta$, let $n_p$ be the residue degree 
of $p$ and $g_p$ the number of ${\mathfrak p} \vert p$, then let $o_{\mathfrak p}(\eta)$ 
and $o_p(\eta)$ be the orders of $\eta$ modulo ${\mathfrak p}$ and $p$, respectively.
Using Frobenius automorphisms, we show that for all $p \gg 0$, 
some explicit divisors of $p^{n_p}-1$ cannot realize $o_{\mathfrak p}(\eta)$ nor $o_p(\eta)$,
and we give a lower bound of $o_p(\eta)$.
Then we obtain that, for all $p \gg 0$ such that $n_p >1$, 
${\rm Prob}(o_p(\eta)<p) \leq \frac{1}{p^{g_p\,(n_p-1) - \varepsilon}}$, where  $\varepsilon = 
O \big( \frac{1}{{\rm log}_2(p)}\big)$;  under the Borel--Cantelli heuristic, 
this leads to $o_p(\eta)>p$ for all $p \gg 0$ such that $g_p(n_p-1) \geq 2$, which covers the
``limit'' cases of cubic fields with $n_p=3$ and quartic fields with $n_p=g_p=2$, but not the 
case of quadratic fields with $n_p=2$. In the quadratic case, the natural conjecture is, 
on the contrary, that $o_p(\eta) < p$ for infinitely many inert $p$. 
Some computations are given with PARI programs.
\end{abstract}

\begin{altabstract} Soit $K/\Q$ Galoisienne, et soit $\eta\in K^\times$ de conjugu\'es 
multiplicativement ind\'ependants.
Pour un premier $p$, non rami\-fi\'e, \'etranger \`a $\eta$, soient
$n_p$ le degr\'e r\'esiduel de $p$ et $g_p$ le nombre de ${\mathfrak p} \vert p$,
puis $o_{\mathfrak p}(\eta)$ et $o_p(\eta)$ les ordres de $\eta$ modulo 
${\mathfrak p}$ et $p$ respectivement.
En utilisant les automorphismes de Frobenius, nous montrons que pour tout $p \gg 0$,
certains diviseurs explicites de $p^{n_p}-1$ ne peuvent r\'ealiser ni $o_{\mathfrak p}(\eta)$ ni $o_p(\eta)$,
et nous donnons une borne inf\'erieure de $o_p(\eta)$.
Ensuite nous obtenons que ${\rm Prob}(o_p(\eta)<p) \leq \frac{1}{p^{g_p\,(n_p-1) - \varepsilon}}$,
o\`u $\varepsilon = O \big( \frac{1}{{\rm log}_2(p)}\big)$, pour tout $p \gg 0$ tel que $n_p >1$~; 
sous l'heuristique de Borel--Cantelli, ceci conduit \`a
$o_p(\eta)>p$ pour tout $p \gg 0$ tel que $g_p(n_p-1) \geq 2$, ce qui couvre les cas ``limites''
des corps cubiques avec $n_p=3$ et des corps quartiques avec $n_p=g_p=2$,
mais non celui des corps quadratiques avec $n_p=2$. Dans le cas quadratique, la conjecture 
naturelle est, au contraire, que  $o_p(\eta) < p$ pour une infinit\'e de $p$ inertes. 
Des calculs sont donn\'es via des programmes PARI.
\end{altabstract}

\maketitle

\section{Frobenius automorphisms} \label{sect1}
\subsection{Generalities} 
Let $K/\Q$ be Galois of degree $n$, of Galois group~$G$.  
Denote by $h$ a possible residue degree of an unramified prime ideal of $K$, that is to say a divisor of $n$ 
for which there exists a {\it cyclic} subgroup $H$ of $G$ of order $h$. Indeed, one knows that, for any generator 
$s$ of $H$, there exist infinitely many prime numbers $p$, unramified in $K/\Q$, such that $s$ is the Frobenius 
$s_{\mathfrak p}$ of a prime ideal ${\mathfrak p} \mid p$ in $K/\Q$~; $H$ is then the 
decomposition group $H_{\mathfrak p}$ of ${\mathfrak p}$.
Reciprocally, any unramified ${\mathfrak p}$ has a cyclic decomposition group $H_{\mathfrak p}$ 
with a canonical generator $s_{\mathfrak p}$ (the Frobenius). 

\medskip
Of course, if $s_1$ and $s_2$ are two distinct generators of $H$, the sets of corresponding primes 
$p$ are disjoint (e.g., take the cyclotomic field $K=\Q(\zeta_5)$ of fifth roots of unity and $H=G$
(cyclic of order $h=4$), with $\zeta_5^{s_1} = \zeta_5^2$ and $\zeta_5^{s_2} = \zeta_5^3$; this characterizes the 
sets $\{p: \   p \equiv 2 \pmod 5\}$ and $\{p : \   p \equiv 3 \pmod 5\}$, respectively); see, e.g., 
 \cite[Section 3 of Chapter 7]{Nar}, \cite[Section 3 of Appendix]{Wa}, or 
 \cite[Sections 1.1, 1.2, 4.6, of Chapter II]{Gr3} for Chebotarev's density theorem 
 and properties of Frobenius automorphisms. 

\smallskip
But we consider such a fixed residue degree $h \mid n$ since we shall see that the statement of 
our main result, on the order of $\eta\in K^\times$ modulo a prime ideal 
${\mathfrak p}$, does not depend on the conjugate of the decomposition group $H_{\mathfrak p}$ 
of ${\mathfrak p}$, nor on its Frobenius $s_{\mathfrak p}$, but only on the residue degree $n_p$ 
of the corresponding prime number $p$ under ${\mathfrak p}$ (in other words, we shall 
classify the set of unramified prime ideals ${\mathfrak p}$ of $K$ by means of the sole criterion $n_p=h$; 
so, any of the ${\mathfrak p}$, with $n_p=h$, would have the common property, depending on $h$, 
given by our main theorem). 

\smallskip
Since the $h \mid n$ are finite in number, everything is effective (e.g., $h\in \{1, 2, 3\}$ for a dihedral 
group $G=D_6$, but $h\in \{1, 2\}$ for the Galois group of any compositum of quadratic fields).

\subsection{Orders modulo ${\mathfrak p}$ and modulo $p$}\label{order}
Let $\eta \in K^\times$. In the sequel we shall assume that  the multiplicative $\Z[G]$-module 
$\langle \eta \rangle_G^{}$ generated by $\eta$ is of $\Z$-rank $n$ 
(i.e., $\langle \eta \rangle_G^{} \otimes \Q \simeq \Q[G]$), 
but this is not needed for the following definition.

\begin{definition} \label{ordres} {\rm Let $p$ be a prime number, unramified in $K/\Q$, prime to $\eta$,
and let ${\mathfrak p}$ be a prime ideal of $K$ dividing $p$.

\smallskip \noindent
We define the {\it order of $\eta$ modulo ${\mathfrak p}$} (denoted $o_{\mathfrak p}(\eta)$) 
to be the least nonzero integer $k$ such that $\eta^k \equiv 1 \pmod {\mathfrak p}$. 

\smallskip \noindent
We define the {\it order of $\eta$ modulo $p$} (denoted $o_p(\eta)$) to be the least 
nonzero integer $k$ such that $\eta^k \equiv 1 \pmod p$. }
\end{definition}

\noindent
Of course, $o_{\mathfrak p} (\eta)$ and $o_p (\eta) = 
{\rm lcm\,}(o_{\mathfrak p} (\eta), \,{\mathfrak p} \mid p)$ divide $p^{n_p} - 1$, 
where $n_p$ is the residue degree of $p$ in $K/\Q$,
but we intend to prove (see Theorem \ref{thm} for a more complete and general statement):

\medskip\noindent
{\it Let $h \mid n$ be a possible residue degree in $K/\Q$. Let $\eta \in K^\times$ be such that the 
multi\-plicative $\Z[G]$-module generated by $\eta$ is of $\Z$-rank $n$. 
Then for all large enough prime $p$ (denoted $p \gg 0$ in all the paper)
with residue degree $n_p = h$, the orders $o_{\mathfrak p} (\eta)$ for any 
${\mathfrak p}\mid p$, and $o_p (\eta)$ do not divide 
any of the integers 

\centerline{$\ds D_{h, \delta}(p) := \frac{p^h -1} {\Phi_\delta(p)}$, $\ \,\delta \mid h$,}

\smallskip\noindent
where $\Phi_\delta(X)$ is the $\delta$th cyclotomic polynomial. }

\bigskip 
Consider, for any unramified prime $p$, the characteristic property of the Frobenius 
automorphism $s_{\mathfrak p}$ of ${\mathfrak p} \mid p$ in $K/\Q$,
$$\eta^{s_{\mathfrak p}} \equiv \eta^p \pmod {\mathfrak p}. $$

\noindent
Let $H_{\mathfrak p} := \langle s_{\mathfrak p} \rangle$ be the decomposition group of ${\mathfrak p}$ 
(denoted $H$ to simplify) and let $\sigma \in G/H$ (or a repre\-sentative in $G$); the Frobenius 
$s_{{\mathfrak p}^\sigma}$ of ${\mathfrak p}^\sigma$ is 
$s_{\mathfrak p}^\sigma := \sigma \cdot s_{\mathfrak p}\cdot   \sigma^{-1}$ and we get 
$\eta^{s_{{\mathfrak p}^\sigma}} \equiv \eta^p \pmod{ {\mathfrak p}^\sigma}$. 
So, if $s_{\mathfrak p}$ and $\sigma$ commute this leads to $s_{\mathfrak p}^\sigma = 
s_{{\mathfrak p}^\sigma} = s_{\mathfrak p}$ and $\eta^{s_{\mathfrak p}} 
\equiv \eta^p \pmod{ {\mathfrak p}^\sigma}$.  In other words, we have
$$\eta^{s_{\mathfrak p}} \equiv \eta^p \ \ \  
\Big(\hbox{mod }{\prd_{{\substack{\sigma \in G/H \\ \sigma.s_{\mathfrak p} = s_{\mathfrak p}.\sigma}} } 
{\mathfrak p}^\sigma}\Big).$$
In the Abelian case, we get (independently of the choice of ${\mathfrak p} \mid p$)
$$\eta^{s_{\mathfrak p}} \equiv \eta^p \pmod p . $$ 

\begin{lemma}\label{poly} Let $\eta \in K^\times$ be such that the multiplicative $\Z[G]$-module
$\langle \eta \rangle_G^{}$ is of $\Z$-rank $n$ and let $\mu(K)$ be the group of roots of unity of $K$.
Let $H$ be a cyclic subgroup of $G$ and let $s$ be any generator of $H$; let $f(X) \in \Z[X]$ be a given 
polynomial such that $f(s) \ne 0$ in $\Z[H]$.

\smallskip\noindent
Then, for all prime numbers $p \gg 0$ such that there exists a prime ideal
${\mathfrak p} \mid p$  for which $s_{\mathfrak p}=s$, whenever $\zeta \in \mu(K)$ we have 
$$\eta^{f(p)} \not\equiv \zeta \pmod {\mathfrak p}.$$ 
\end{lemma}

\begin{proof} We have $\eta^{f(p)} \equiv \eta^{f(s)} \pmod  {\mathfrak p}$; thus, if $\eta^{f(p)} \equiv 
\zeta \pmod {\mathfrak p}$ for some $\zeta$, this leads to $\eta^{f(s)}  - \zeta \equiv 0 \pmod {\mathfrak p}$ 
giving, by the norm in $K/\Q$,
$${\rm N}_{K/\Q}(\eta^{f(s)}  - \zeta) \equiv 0 \pmod {p^{\vert H \vert}} . $$ 
Since  $\langle \eta \rangle_G^{}$ is of multiplicative $\Z$-rank $n$ and $f(s) \ne 0$, 
we have $\eta^{f(s)} \notin \mu(K)$; then ${\rm N}_{K/\Q}(\eta^{f(s)}  - \zeta)$ is a nonzero rational constant 
depending only on $\eta$, $f(s)$, $\zeta$, and whose numerator is in $p^{\vert H \vert}\Z$ 
(a contradiction for $p \gg 0$).
\end{proof}

The statement of the lemma does not depend on the choice of $s$ generating $H$, nor
on the choice of the prime ideal ${\mathfrak p} \mid p$ such
that $s_{\mathfrak p} = s$ (in the Abelian case, any ${\mathfrak p} \mid p$ is suitable since 
$s_{{\mathfrak p}^\sigma}=s_{\mathfrak p}$ for all $\sigma \in G$).

\smallskip
If $s \in G$ is of order $h \geq 1$, any nonzero element of $\Z[H]$ can be writen $f(s)$ where 
$f(X)$ is of degree $<h$; if we take $f(X)$ of degree $0$, then we have
$f(s)=f \in \Z\setminus \{0\}$ 
regardless of $h$ and $s$, giving the obvious result

\medskip
\centerline{$\eta^f \not\equiv \zeta \pmod {\mathfrak p}\ $ for any $p \gg 0$.}

\medskip
Naturally, an interesting application of this Lemma is when $f(X) \mid X^h-1$ in $\Z[X]$,
$f (X) \ne X^h-1$, and when the degree of $f(X)$ is maximal. This explains why the case 
$h=n_p=1$ ($p$ totally split in $K/\Q$) is uninteresting since $f(X) \mid X-1$, with $f (X) \ne X-1$, 
gives $f=1$ and the same conclusion as above.

\section{Consequences for the values of $o_{\mathfrak p}(\eta)$ and $o_p(\eta)$} \label{sect2}
We have the factorization 
$$p^h-1 =\prd_{\delta \mid h} \Phi_\delta(p),$$ 
where $\Phi_\delta(X)$ is the $\delta$th cyclo\-tomic
polynomial (see \cite[Chapter 2]{Wa}). So we can consider the divisors $\ds\prd_{\delta \in I} \Phi_\delta(p)$,
where $I$ is any strict subset of the set  of divisors of $h$.
Of course, it will be sufficient to restrict ourselves to {\it maximal} subsets $I$, which gives the divisors 
$\ds D_{h, \delta}(p) :=\frac{p^h-1}{\Phi_{\delta}(p)}, \ \ \delta \mid h$.
For instance, if $h=6$, we get the set
$$\{ p^5 + p^4 + p^3 + p^2 + p + 1 , \  p^5 - p^4 + p^3 - p^2 + p - 1, \ 
 p^4 - p^3 + p - 1, \  p^4 + p^3 - p - 1 \}, $$

\noindent
giving the complete set of ``polynomial divisors'' of $p^6-1$,

\medskip \noindent
$\{1, p - 1, p + 1, p^2 - 1, p^2 - p + 1, p^3 - 2\,p^2 + 2\,p - 1, p^3 + 1, p^4 - p^3 + p - 1, p^2 + p + 1$,

\smallskip \noindent
$p^3 - 1, p^3 + 2\,p^2 + 2\,p + 1, p^4 + p^3 - p - 1, p^4 + p^2 + 1, p^5 - p^4 + p^3 - p^2 + p - 1$, 

\hfill $p^5 + p^4 + p^3 + p^2 + p + 1\}$. 

\begin{theorem}\label{thm} Let $K/\Q$ be Galois of degree $n$, of Galois group~$G$. 
Let $h \mid n$ be a possible residue degree in $K/\Q$. Let $\mu(K)$ be the group of roots of unity contained in $K$.
Let $\eta \in K^\times$ be such that the multiplicative $\Z[G]$-module generated by $\eta$ is of $\Z$-rank $n$. 

\smallskip\noindent
Then for all (unramified) prime number $p \gg 0$, with residue degree $n_p = h$, and for any 
prime ideal ${\mathfrak p}\mid p$, the least integer $k \geq 1$ for which there exists $\zeta \in \mu(K)$ 
such that $\eta^k \equiv \zeta \pmod {\mathfrak p}$ is a divisor of $p^h-1$ which does not divide any 
of the integers 
$$D_{h, \delta}(p) := \frac{p^h -1} {\Phi_\delta(p)},\ \ \delta \mid h,$$ 
where $\Phi_\delta(X)$ is the $\delta$th cyclotomic polynomial.

\smallskip\noindent
Hence $o_{\mathfrak p} (\eta)$ and a fortiori $o_p (\eta)$
(cf. Definition \ref{ordres}), do not divide any of the $D_{h, \delta}(p)$.
\end{theorem} 

\begin{proof} Let $k'={\rm gcd}\,(k, p^h-1)$. Then we have $k'=\lambda\,k  + \mu\,(p^h-1)$,
$\lambda, \mu \in \Z$, and $\eta^{k'}\equiv \eta^{\lambda\,k}
\equiv \zeta^\lambda =: \zeta' \pmod {\mathfrak p}$; but $k' \mid k$, so $k = k' \mid p^h-1$.
Suppose that $k$ divides some 
$\ds D_{h, \delta}(p) =\frac{p^h -1} {\Phi_\delta(p)} = \prd_{\delta' \mid h,\, \delta' \ne \delta} \Phi_{\delta'}(p)$.
Let $s$ be the Frobenius of ${\mathfrak p}$ and $H=\langle s \rangle$ its decomposition 
group (of order $h$). Thus $\eta^k \equiv \zeta \pmod {\mathfrak p}$ yields 
$$\eta^{D_{h, \delta}(p)} \equiv \zeta' \pmod {\mathfrak p}, \  \hbox{ $\zeta' \in \mu(K)$,} $$
giving 
$$\eta^{D_{h, \delta}(p)} \equiv \eta^{D_{h, \delta}(s)} \equiv \zeta' \pmod {\mathfrak p}.$$

\smallskip\noindent
From $\Z[H] \simeq \Z[X]/(X^h-1)\, \Z[X]$, we get in $\Z[H]$
$$D_{h, \delta}(s) = \prd_{\delta' \mid h,\, \delta' \ne \delta} \Phi_{\delta'}(s) \ne 0$$ 
since $D_{h, \delta}(X) \notin (X^h-1)\, \Z[X]$; the polynomial $D_{h, \delta}(X)\in \Z[X]$ being 
independent of $p$,  Lemma \ref{poly} applied to  $f (X)= D_{h, \delta}(X)$ gives a contradiction 
for all $p \gg 0$ with residue degree $n_p = h$.
\end{proof}

If $\langle \eta \rangle_G^{}$ is not of $\Z$-rank $n$, a statement does exist which depends on the 
$G$-representation $\langle \eta \rangle_G^{}$; for instance, let $K=\Q(\sqrt m)$ 
and $\eta \in K^\times \setminus \mu(K)$:

\medskip
-- If ${\rm N}_{K/\Q}(\eta)=\pm 1$, then $o_p(\eta) \nmid D_{2, 2}(p) = p-1$ 
for all prime $p \gg 0$, inert in $K/\Q$.

\medskip
-- If $\eta^{1-s}= \pm1$, 
then $o_p(\eta) \nmid D_{2, 1}(p) = p+1$ (e.g., $\eta = \sqrt m$, $m \ne -1$).

\medskip
The expression {\it ``for all $p \gg 0$ of residue degree $n_p = h$''} in the theorem is effective 
and depends, numerically, only on $h$ and the conjugates of $\eta$. 

\medskip
The theorem gives the generalization of the particular case $h=2$  in \cite{Gr1}.

\medskip
In the above case $h=6$ and $p \gg 0$ (with $n_p=6$), the orders $o_{\mathfrak p} (\eta)$ 
are divisors of $p^6-1$ which are not divisors of any of the integers in the set:
$$\{\,  p^5 + p^4 + p^3 + p^2 + p + 1 , \  p^5 - p^4 + p^3 - p^2 + p - 1, \ 
 p^4 - p^3 + p - 1, \  p^4 + p^3 - p - 1\}. $$

For $p=7$ and $h=6$, we have $60$ divisors of $p^6-1 = 2^4 \cdot 3^2 \cdot 19 \cdot 43$, 
and the distinct divisors of these $4$ polynomials are the $52$ integers:

\medskip\noindent
$1, 2, 3, 4, 6, 8, 9, 12,16, 18, 19, 24,36,  38, 43,48,  57, 72, 76, 86, 114, 129, 144$,  

\smallskip \noindent
$152,  171, 172, 228, 258, 304, 342, 344, 387, 456, 516, 684, 688,774,  817, 912$,

\smallskip \noindent
$1032, 1368, 1634, 2064, 2451,2736,  3268, 4902, 6536,7353,  9804,\! 14706,\! 19608$.

\medskip
So the remaining (possible) divisors of $p^6-1$ are 

\medskip \noindent
$1548$, $3096$, $6192$, $13072$, $29412$, $39216$, $58824$, $117648$.

\medskip
Of course, in our example, the prime $p=7$ is too small regading $\eta$, 
but the interesting fact (which is similar for larger $p$ and any integer $h$) 
is the great number of impossible divisors of $p^h-1$ for small numbers $\eta$.

\noindent
For $p=1093$ (resp. $504202701918008951235073$), only $76$ (resp. $242424$) 
divisors are possible among the $384$ (resp. $518144$) divisors of $p^6-1$.    

\smallskip
The case $h=\ell$ (a prime) implies that $o_{\mathfrak p}(\eta)$ is not a divisor of $p-1$ nor a divisor of
$p^{\ell-1}+ \cdots+ p+1$ for $p \gg 0$ with residue degree $n_p=\ell$; this means that 
$o_{\mathfrak p}(\eta) = d_1 d_2$ 
with $d_1 \mid p-1$, $d_1 \ne 1$, $d_2 \mid p^{\ell-1}+ \cdots+ p+1$, $d_2 \ne 1$ (taking care of the 
fact that when $p\equiv 1 \pmod \ell$, we have the relation ${\rm gcd\,}(p-1, p^{\ell-1}+ \cdots+ p+1) = \ell$). 

\begin{remark}\label{probas} {\rm 
It is clear that if $r \in \N \setminus \{0\}$ is small, Theorem \ref{thm} implies
that for all prime $p \gg 0$ with residue degree $n_p = h$ and for any ${\mathfrak p}\mid p$, 
the least integer $k \geq 1$ for which there exists $\zeta \in \mu(K)$ such that 
$\eta^k \equiv \zeta \pmod {\mathfrak p}$ cannot divide any of the integers $r \cdot D_{h, \delta}(p)$, 
$\delta \mid h$ (indeed, $\eta^r$ is still small in an Archimedean point of view). 
This makes sense only when $r=r_\delta$ is choosen, for each $D_{h, \delta}(p)$, as a small divisor of
$\Phi_\delta(p)$.

\smallskip\noindent
So the probability of $o_{\mathfrak p}(\eta) \mid r_\delta \cdot  D_{h, \delta}(p)$ 
increases (from $0$ to $1$) when the factor $r_\delta \mid \Phi_\delta(p)$ increases 
(from $r_\delta=1$ to $r_\delta=\Phi_\delta(p)$).
In the example $h=\ell$, where $o_{\mathfrak p}(\eta) = d_1 d_2$,
$d_1 \mid p-1$, $d_2 \mid p^{\ell-1}+ \cdots+ p+1$, we have $d_1 \ \&\  d_2 \to \infty$ for $p\to \infty$.}
\end{remark}

\section{A numerical example}\label{sect3}
Let $K = \Q(x)$ be the cyclic cubic field of conductor $7$ defined by $x=\zeta_7+\zeta_7^{-1}$
from a primitive seventh root of unity $\zeta_7$; its irreducible polynomial is $X^3+X^2-2X-1$. 

\smallskip
Let $\eta=8\,x+5$ of norm $-203$; then for $p<200$, inert in $K$ (i.e., $p^2 \not\equiv 1 \pmod 7$), 
we obtain the exceptional example 
$o_{17}(\eta) = 307 = p^2+p+1$ and no other when $p$ increases; we get some illustrations with a small 
$r \mid p-1$, $r >1$ (e.g., $p=101$, $r=2$, with $o_{p}(\eta) =r \cdot ( p^2+p+1)$),
according to the following numerical results; note that when $p\equiv 1 \pmod 3$, we have 
$o_p(\eta)=\frac{1}{3} \cdot {\rm gcd\,}(o_p(\eta), p-1) \cdot {\rm gcd\,}(o_p(\eta), p^2+p+1)$:

\medskip
(i) $p \equiv 2 \pmod 7$:
$$\begin{array}{cccc}
\hbox{$p$} &\hbox{\hspace{0.5cm}${\rm gcd\,}(o_p(\eta), p-1)$} & \hbox{\hspace{0.8cm} 
${\rm gcd\,}(o_p(\eta), p^2+p+1)$} \\ \\
2  &     1  & 1 \\
23  &     11  & 553 \\
37    &   36  &  201 \\
79   &    78   &   6321 \\
107   &    53   &   11557 \\
149  &     37   &   22351 \\
163  &     54   &   26733 \\
191  &     190   &   36673
\end{array} $$

(ii) $p \equiv 3 \pmod 7$:
$$\begin{array}{cccc}
\hbox{$p$}   &  \hbox{${\rm gcd\,}(o_p(\eta), p-1)$}  &  \hbox{\hspace{0.8cm} 
${\rm gcd\,}(o_p(\eta), p^2+p+1)$} \\ \\
3  &    1  & 13 \\
\ 17^{\,*}    &   1  &307 \\
31   &    15 & 993 \\
59   &    58 & 3541 \\
73   &    9 & 5403 \\
\ \ \  101^{\,**}   &    2 & 10303 \\
157   &    26 & 8269 \\
199  &     198 & 39801
\end{array} $$

(iii) $p \equiv 4 \pmod 7$:
$$\begin{array}{cccc}
\hbox{$p$}   &  \hbox{\hspace{0.5cm}${\rm gcd\,}(o_p(\eta), p-1)$}  &  \hbox{\hspace{0.8cm} 
${\rm gcd\,}(o_p(\eta), p^2+p+1)$} \\ \\
11  &     10 & 133 \\
53    &   26 & 2863 \\
67   &    33 & 4557 \\
109  &     27 & 11991 \\
137   &    136 & 18907 \\
151   &    75 & 22953 \\
179   &    89 & 32221 \\
193   &    192 & 37443
\end{array} $$

(iv) $p \equiv 5 \pmod 7$:
$$\begin{array}{cccc}
\hbox{$p$}   &  \hbox{\hspace{0.5cm}${\rm gcd\,}(o_p(\eta), p-1)$}  &  \hbox{\hspace{0.8cm} 
${\rm gcd\,}(o_p(\eta), p^2+p+1)$} \\ \\
5  &     4  & 31 \\
19   &    9 & 381 \\
47   &    23 & 2257 \\
61   &    10 & 1261 \\
89   &    11 & 8011 \\
103   &    102 & 10713 \\
131   &    65 & 17293 \\
173   &    172 & 30103
\end{array} $$

With the same data, the least values of ${\rm gcd\,}(o_p(\eta), p-1)$ are:

\smallskip\noindent
1 (for $p=2$, $3$, $17$),  
2 (for $p=101$), 3 (for $p=13669$, for wich we get $o_p(\eta) = 560565693$), 4 (for $p=5$, $317$), 
9 (for $p=19$, $73$). 

\smallskip \noindent
Up to $p \leq 10^7$, we have no other solutions for ${\rm gcd\,}(o_p(\eta), p-1) < 10$. 

\smallskip \noindent
For ${\rm gcd\,}(o_p(\eta), p^2+p+1) < 100$ we get $1$ (for $p=2$), $13$ (for $p=3$), $31$ (for $p=5$); 
for ${\rm gcd\,}(o_p(\eta), p^2+p+1) < 1000$ we only have the primes $p=2, 3, 5, 11, 17, 19, 23, 31, 37$ 
giving a solution up to $10^7$.

\section{A lower bound for $o_p(\eta)$}\label{sect4}

When $\eta$ is fixed in $K^\times$,  very small orders are impossible as $p \to \infty$ because of the 
following theorem giving Archimedean constraints; in this result none hypothesis is done on the 
rank of the multiplicative $\Z[G]$-module generated by $\eta$ (except that this $\Z$-rank is assumed 
to be $\ne 0$) nor on the field $K$ itself. We denote by $Z_K$ the ring of integers of $K$.

\begin{theorem} \label{log} Let $\mu(K)$ be the group  of roots of unity contained in $K$.
Let $\eta \in K^\times \setminus \mu(K)$ and $\nu \in \N \setminus \{0\}$ be such that $\nu\,\eta \in Z_K$. 
Then, for any $p$ prime to $\eta$ and $\nu$, the congruence 
$\eta^k \equiv \zeta \pmod p, \,  \hbox{$\zeta \in \mu(K)$,  $k \geq 1$}$,
implies the inequality
$$\ds k \geq \frac{{\rm log}(p) - {\rm log}(2) }{{\rm max} \big ({\rm log}(\nu \cdot c_0(\eta)) , {\rm log}(\nu) \big)}, $$
where $c_0(\eta) = {\rm max}_{\sigma \in G}(\vert \eta^\sigma \vert \big )$.

\noindent
If $\eta \in Z_K$ (i.e., $\nu=1$), then we get $\ds k \geq \frac{{\rm log} (p-1) }{{\rm log} (c_0(\eta))}$.

\smallskip \noindent
In other words, if $Z_{K, (p)}$ is the ring of $p$-integers of $K$, the order of the image of $\eta$ in 
$Z_{K, (p)} \big / \mu(K)\!\cdot \! (1 + p\,Z_{K, (p)})$, and a fortiori $o_p(\eta)$, satisfies the above inequalities.
\end{theorem}

\begin{proof} Put $\eta = \frac{\theta}{\nu}$, with $\theta \in Z_K$.
The congruence is equivalent to $\theta^k = \zeta \, \nu^k  + \Lambda\, p$, where 
$\Lambda \in Z_K \setminus \{0\}$ (because $\eta \notin \mu(K)$).
Taking a suitable conjugate of this equality, we can suppose $\vert \Lambda \vert \geq 1$.
Thus 
$$\vert \Lambda \vert\, p= \vert \theta^k - \zeta\, \nu^k \vert \leq \vert \theta \vert^k + \nu^k$$ 
giving $\vert \theta \vert^k+\nu^k \geq p$; so, using a conjugate $\theta_0$ such that 
$\vert\theta_0\vert = {\rm max}_{\sigma \in G}(\vert \theta^\sigma \vert )$, 
we have a fortiori $\vert \theta_0 \vert^k+\nu^k \geq p$, with $\vert \theta_0 \vert \geq 1$ 
since $\theta \in Z_K$.

\smallskip
(i) If $\nu \geq 2$, then 
$$p \leq \vert \theta_0 \vert^k+\nu^k \leq 2\, {\rm max}(\vert \theta_0 \vert^k, \nu^k)$$
giving the result.

\smallskip
(ii) The case $\nu = 1$, used in \cite[Lemme 6.2]{Gr1}, gives $\vert \theta_0 \vert^k  \geq p-1$, 
hence the upper bound $\ds k \geq \frac{{\rm log}(p-1) }{ {\rm log}(c_0(\eta))}$ since 
$\vert \theta_0 \vert=c_0(\eta)>1$ (because $\eta \notin \mu(K)$).
\end{proof}

Under the assumptions of Theorem \ref{thm} we have the following result.

\begin{corollary} Suppose to simplify that $\eta \in Z_K$; let $p$ be unramified of residue degree 
$n_p$ such that for some $\delta \mid n_p$, $o_p(\eta) = r \cdot d$, $d \mid D_{n_p, \delta}(p)$, 
$r = r_\delta \mid \Phi_\delta(p)$ (cf. Remark \ref{probas}). Then 
$\ds r \geq \frac{{\rm log} (p-1) }{{\rm log} (c_0(\eta^{D_{n_p, \delta}(s)}))}$, 
where $s$ generates any decomposition group of $p$.
\end{corollary}

In the previous example of Section \ref{sect3}, for $p \approx 10^7$ and $o_p(\eta)=r\cdot d$, $d \mid D_{3, \delta}(p)$,
we find, from the corollary, $r \geq 3$ for $\delta = 1$ (i.e., $r \mid p-1$), $r \geq 9$ for $\delta = 3$ 
(i.e., $r \mid p^2+p+1$).

\section{Densities--Probabilities for $o_{\mathfrak p}(\eta)$ and $o_p(\eta)$} \label{sect5}
In this section, we examine some probabilistic aspects concerning the orders modulo ${\mathfrak p} \mid p$ 
of an $\eta \in K^\times$. For any $p$, unramified in $K/\Q$, recall that 
$g_p$ is the number of prime ideals ${\mathfrak p} \mid p$ and $n_p$ the common residue degree of these ideals. 
Let $Z_K$ be the ring of integers of $K$; the residue fields $F_{\mathfrak p} = Z_K/{\mathfrak p}$ 
are isomorphic to $\F_{p^{n_p}}$.

\subsection{Densities}\label{D}
It is assumed in this short subsection that $p$ is fixed and that $\eta \in K^\times$ is a variable modulo $p$, 
prime to the given~$p$; in other words, $\eta$ varies in the group $(Z_{K, (p)} / p\, Z_{K, (p)})^\times$
of invertible elements of the quotient $Z_{K, (p)} / p\, Z_{K, (p)}$, where $Z_{K, (p)}$ is the ring of 
$p$-integers of $K$, so that we have
$$(Z_{K, (p)} / p\, Z_{K, (p)})^\times \simeq \prd_{{\mathfrak p} \mid p} F_{\mathfrak p}^\times
\ \  \hbox{($p$ unramified).} $$ 
For each prime ideal ${\mathfrak p} \mid p$, let 
$\eta_{\mathfrak p} \in F_{\mathfrak p}^\times$ be the residue image of $\eta$ at ${\mathfrak p}$. 

\smallskip
The density of numbers $\eta$, whose diagonal image is given in 
$\ds \prd_{{\mathfrak p} \mid p} F_{\mathfrak p}^\times$, is \vspace{-0.4cm}
$$\frac{1}{(p^{n_p}-1)^{g_p} }$$ 
because the map $\eta \pmod p \mapsto (\eta_{\mathfrak p})_{{\mathfrak p}\mid p} \in
\prd_{{\mathfrak p} \mid p} F_{\mathfrak p}^\times$ yields an isomorphism (from chinese remainder theorem) 
and, in some sense, the $g_p$ conditions on the $\eta_{\mathfrak p}$, ${\mathfrak p} \mid p$, 
are independent as $\eta$ varies (the notion of density is purely algebraic and the previous Archimedean 
obstructions of Sections \ref{sect2}  \& \ref{sect4} do not exist). 
Thus the orders $o_{\mathfrak p}(\eta)$ and $o_p(\eta)$ have canonical densities
(see \S\,\ref{phi}).

\subsection{Probabilities and Independence}\label{PI}
We shall speak of probabilities when, on the contrary, $\eta \in K^\times \setminus \mu(K)$ is fixed and 
when $p\to\infty$ is the variable; but to avoid trivial cases giving obvious obstructions (as $\eta\in\Q^\times$ 
for which $o_p(\eta) \mid p-1$ for any $p$ regardless of the residue degree of $p$; see \S\,\ref{obstructions} 
for more examples), we must put some assumptions on $\eta$ so that $o_p(\eta)$ can have any 
{\it possible value} dividing $p^{n_p}-1$ (by reference to Theorem \ref{thm}, Remark \ref{probas}, 
and Theorem \ref{log} giving moreover theoretical limitations for the orders, 
so that the true probabilities are significantly lower). 

\medskip
Let $H$ be the decomposition group of a prime ideal ${\mathfrak p}_0 \mid p$, $p$ unramified in $K/\Q$. 
Considering $F_{{\mathfrak p}_0}^\times$ as a $H$-module ($H$ is generated by the global 
Frobenius $s=s_{{\mathfrak p}_0}$ which by definition makes sense in $F_{{\mathfrak p}_0}/ \F_p$), 
$\prd_{{\mathfrak p} \mid p} F_{\mathfrak p}^\times$ is the induced representation and we get 
$\prd_{{\mathfrak p} \mid p} F_{\mathfrak p}^\times = \plus_{\sigma \in G/H} \sigma F_{{\mathfrak p}_0}^\times$
where $\sigma F_{{\mathfrak p}_0}^\times = F^\times_{{{\mathfrak p}^\sigma_0}}$ for all $\sigma \in G/H$
(using additive notation for convenience).

\medskip\noindent
In the same way, the representation $\langle \eta \rangle_G^{}$ can be written
$\langle \eta \rangle_G^{} = \!\!\sm_{\sigma \in G/H} \!\!\sigma  \langle \eta \rangle_H^{}$,
where $\langle \eta \rangle_H^{}$ is the multiplicative $\Z[H]$-module generated by $\eta$. 
So, for natural congruential reasons (that must be valid regardless of the prime $p$) concerning the map 
$\eta \!\pmod p \mapsto (\eta_{\mathfrak p})_{{\mathfrak p} \mid p}$, the representation 
$\langle \eta \rangle_G^{}$ must be induced by the $H$-representation $\langle \eta \rangle_H^{}$, i.e., we must have
$$\langle \eta \rangle_G^{} = \plus_{\sigma \in G/H} \sigma \langle \eta \rangle_H^{}$$
(otherwise, any nontrivial $\Z$-relation between the conjugates of $\eta$ will give non-independent variables 
$\eta_{\mathfrak p}$ in a probabilistic point of view). 
Since any cyclic subgroup $H$ of $G$ is realizable as a decomposition group when $p$ varies, 
the above must work for any $H$; taking $H=1$, we obtain that 
$\langle \eta \rangle_G^{} = \plus_{\sigma \in G} \langle \eta^\sigma \rangle_\Z^{}$
which is equivalent for $\langle \eta \rangle_G^{}$ to be of $\Z$-rank $n$, giving the following heuristic
in relation with the properties of the normalized $p$-adic regulator of $\eta$ studied in \cite{Gr1}.

\begin{heuristic} \label{heur1} {\it Let $K/\Q$ be Galois of degree $n$, of Galois group~$G$. 
Consider $\eta \in K^\times$ and, for any prime number $p \gg 0$, unramified in $K/\Q$ 
and prime to~$\eta$, let $(\eta_{\mathfrak p})_{{\mathfrak p} \mid p}$ be the diagonal 
image of $\eta$ in $\prod_{{\mathfrak p} \mid p} F_{\mathfrak p}^\times$. 

\noindent
The components $\eta_{\mathfrak p}$ are {\it independent}, in the meaning that for given 
$a_{\mathfrak p} \in F_{\mathfrak p}^\times$,
$${\rm Prob} \big (\eta_{\mathfrak p}= a_{\mathfrak p},\ \forall \,{\mathfrak p}\mid p \big) = 
\hbox{$\prod_{{\mathfrak p}\mid p}$}\  {\rm Prob} \big(\eta_{\mathfrak p}= a_{\mathfrak p} \big), $$
if and only if $\eta$ generates a multiplicative  $\Z[G]$-module of $\Z$-rank $n$.}
\end{heuristic}

\subsection{Remarks and examples}\label{obstructions}
Suppose that $\eta$ generates a multiplicative $\Z[G]$-module of $\Z$-rank $n$, which has obvious 
consequences (apart the fact that $\eta \notin \mu(K)$):

\smallskip
(i) This implies that $\eta$ is not in a strict subfield $L$ of $K$; otherwise, if $H$ is a non-trivial cyclic
subgroup of $G$ (hence of order $h>1$) such that $L \subseteq K^H$, for any unramified prime $p$ 
such that $H$ is the decomposition group of ${\mathfrak p}\mid p$ with Frobenius~$s$, 
$o_{\mathfrak p}(\eta)$ is not a random divisor of $p^{h}-1$ but  a divisor of $p-1$, 
the residue field of $K^H$ at ${\mathfrak p}$ being $\F_p$  for infinitely many~$p$.

\smallskip
(ii) In the same way, $\eta$ cannot be an element of $K^\times$ of relative norm $1$ in 
$K/K^H$, $H \ne 1$, because of the relation ${\rm N}_{K/K^H}( \eta) = 1$ giving 
$$\eta^{p^{h-1}+ \cdots + p+1} \equiv 1 \pmod {{\mathfrak p}}. $$
For the unit $\eta = 2\sqrt 2+3$ and any $p$ inert in $\Q(\sqrt 2)$, we obtain 
$\eta^{p+1} \equiv 1 \pmod p$ (i.e., $o_p(\eta) \mid p+1$), giving infinitely many $p$ 
such that $o_p(\eta)<p$:

\smallskip\noindent
$(p, o_p(\eta)) =$ 

\smallskip \noindent
$(29,10), (59,20), (179,36), (197,18), (227,76), (229,46), (251,84), (269,30)$, 

\smallskip \noindent
$(293,98), (379,76), (389,78), (419,140), (443,148), \ldots$

\bigskip
(iii) Let $K=\Q(j, \sqrt[3] 2)$, where $j$ is a primitive third root of unity, and let $\eta = \sqrt[3] 2 -1$ 
(a unit of $\Q(\sqrt[3] 2)$); for the same reason with $H = {\rm Gal}(K/\Q(j))$,
from $\eta^{s^2+s+1} =1$, we get, for any prime $p$ inert in $K/\Q(j)$, 
$$\eta^{p^2+p+1} \equiv 1\! \pmod p$$ 
(for $p=7$, $\eta$ is of order $19$ modulo~$p$ and we have infinitely many $p$ such that 
$o_p(\eta) \mid p^2+p+1$). 

\smallskip
In such a non-Abelian case, some relations of dependence can also occur in a specific 
component $F_{\mathfrak p}^\times$, ${\mathfrak p} \mid p$.
Since $\eta= \sqrt[3] 2 -1 \in \Q(\sqrt[3] 2)$, for any ${\mathfrak p}$ inert in $K/\Q(\sqrt[3] 2)$ 
(i.e., $K^H = \Q(\sqrt[3] 2)$, in which case, $p$ splits in $K/\Q(j)$), there exists
a rational $a$ such that $\sqrt[3] 2 \equiv a \pmod {\mathfrak p}$, $\sqrt[3] 2 \equiv 
a \,j^2 \pmod {{\mathfrak p}^s}$, $\sqrt[3] 2 \equiv a\,j \pmod {{\mathfrak p}^{s^2}}$. 
So $\eta \equiv a-1 \pmod {\mathfrak p}$ and $o_{\mathfrak p}(\eta) \mid p-1$ which is 
not necessary true for $o_p(\eta)$:
for $p=5$ we have $\sqrt[3] 2 \equiv 3 \pmod {\mathfrak p}$, $\sqrt[3] 2 \equiv 3j^2 \pmod {{\mathfrak p}^s}$, 
$\sqrt[3] 2 \equiv 3j \pmod {{\mathfrak p}^{s^2}}$. 
Then $\eta \equiv 2 \pmod {\mathfrak p}$ is of order $4$ modulo ${\mathfrak p}$, but $\eta \equiv 3j^2-1 
\pmod {{\mathfrak p}^s}$ is of order $8$ modulo~${\mathfrak p}^s$. So $o_p(\eta) = 8$
but we have some constraints on the $\eta_{\mathfrak p}$.

\subsection{Probabilities for the order of $\eta$ modulo $p$}\label{phi}
Now we suppose that the multiplicative $\Z[G]$-module $\langle \eta \rangle_G^{}$ is of $\Z$-rank $n$.

\begin{remark}{\rm
From Theorem \ref{thm}, we know that $o_p(\eta) \nmid D_{n_p, \delta}(p)$ for all $\delta \mid n_p$, when 
$p\to\infty$; in particular, $o_p(\eta) \nmid p- 1$ if we assume $n_p > 1$. 
For this, the hypothesis on the $\Z$-rank of $\langle \eta \rangle_G^{}$ is fundamental. 
In other words, the probability of some (unbounded) orders is zero.  This is strengthened by Remark \ref{probas}.
Moreover, Theorem \ref{log} gives obstructions for very small orders, which decreases the probabilities 
of small orders; the total defect of probabilities is less than $O({\rm log}(p))$ and is to be distributed 
among all orders, which is negligible.
Thus, this favors large orders which are more probable; this goes in the good direction because we shall 
study probabilities of orders $o_p(\eta)$ less than $p$ when $n_p > 1$.}
\end{remark}

 Although the theoretical values of the probabilities are rather intricate,
in a first approach, we can neglect these aspects and give some results in an heuristic 
point of view corresponding to the case where $\eta$ is considered as a variable (so that 
probabilities coincide with known densities) and we use the heuristic that when $\eta$ is fixed 
once for all, probabilities are much lower than densities as $p \to \infty$ as explained in \S\,\ref{PI}.
Furthermore, we shall use rough majorations (except in the quadratic case and $n_p=2$,
where densities are exact).

\smallskip
If $D \mid p^{n_p}-1$, $o_p(\eta) \mid D$ is equivalent to $\eta_{\mathfrak p}^D =1$ for all ${\mathfrak p} \mid p$. 
So we obtain 
$${\rm Prob} \big (o_p(\eta)=D \big ) \leq {\rm Prob}\big (o_p(\eta) \!\mid\! D \big )
= \prd_{{\mathfrak p} \mid p} {\rm Prob}\big (\eta_{\mathfrak p}^D = 1\big ) \  \hbox{(cf. Heuristic \ref{heur1}).}$$ 
Since $F_{\mathfrak p}^\times$ is cyclic of order $p^{n_p}-1$, we get 
$${\rm Prob}\big (\eta_{\mathfrak p}^D = 1\big ) = \sm_{d \mid D}  \frac{\phi(d)}{p^{n_p}-1} =  \frac{D}{p^{n_p}-1},$$
where $\phi$ is the Euler function, and we obtain, for any $D \mid p^{n_p}-1$, 
$${\rm Prob}\big (o_p(\eta)=D\big ) \leq \Big(\frac{D} {p^{n_p}-1} \Big)^{g_p}. $$

\noindent
If $g_p=1$, then $n_p=n$, and we can replace this inequality by 
$${\rm Prob}\big (o_p(\eta)=D\big ) \leq 
{\rm Density}\big (o_p(\eta)=D\big ) = \frac{\phi(D)} {p^{n}-1}.$$
When $g_p >1$, the exact expression is more complicate since $o_p(\eta) = D$ if and only if 
$o_{{\mathfrak p}_0}(\eta_{{\mathfrak p}_0}) = D$ for at least one ${\mathfrak p}_0 \mid p$ and 
$o_{\mathfrak p}(\eta_{\mathfrak p}) \mid D$ for all ${\mathfrak p} \mid p$, 
${\mathfrak p} \ne {\mathfrak p}_0$, but we shall not need it.

\section{Probabilities of orders $o_p(\eta)<p$}\label{sect6}

 Suppose $p \gg 0$, non totally split in $K/\Q$.
 In \cite{Gr1}, the number $\eta$ is a fixed {\it integer} of $K^\times$ and we have to consider the set
$$I_p(\eta) :=\big \{1, [\eta ]_p, \ldots, [\eta^k ]_p, , \ldots , [ \eta^{p-1} ]_p \big\}, $$
where $[\, \cdot \,]_p$ denotes a suitable residue modulo $p\,Z_K$.
We need that $I_p(\eta)$ be a set with $p$ distinct elements, to obtain valuable statistical results 
on the ``local regulators $\Delta_p^\theta(z)$'', $z \in I_p(\eta)$, to strengthen some important heuristics;
this condition is equivalent to $\eta^k \not \equiv 1 \pmod p$ for all $k=1, \ldots, p-1$, hence to $o_p(\eta)>p$.

\medskip
So we are mainely interested by the computation of ${\rm Prob} \big (o_p(\eta)<p \big )$ when $n_p >1$ and
we intend to give an upper bound for this probability. As we know from Theorem \ref{thm}, 
taking the example of quadratic fields we have, for $\langle \eta \rangle_G^{}$ of $\Z$-rank $2$,
$$o_p(\eta) \nmid p-1\ \ \& \ \ o_p(\eta) \nmid p+1, \ \hbox{for $p\to\infty$; } $$
but $o_p(\eta)< p$ remains possible for small divisors $D$ of $p^{2}-1$ (e.g., $\eta=5+ \sqrt{-1}$ 
for which $p=19$ is inert in $\Q(\sqrt{-1})$ and $o_{19}(\eta) = 3 \cdot 5$ whereas $p-1=18$ and $p+1=20$).

\smallskip
Suppose that $K\ne \Q$ and that the residue degree of $p$ is $n_p>1$. Let 
$${\mathcal D}_p := \{D : \  D \mid p^{n_p}-1,  D<p, \ D \nmid D_{n_p, \delta}(p) \  \forall \, \delta \mid n_p\}. $$

Then we consider that we have, for all $p \gg 0$ of residue degree $n_p$, 
the following heuristic inequality (from Theorems \ref{thm}, \ref{log} \& \S\,\ref{phi}):
\begin{equation}\label{proba}
\begin{aligned}
{\rm Prob}\big (o_p(\eta)  < p\big ) & \leq {\rm Prob}\big (o_p(\eta) \in {\mathcal D}_p\big ) \\
& \leq  \sm_{D \in {\mathcal D}_p}\Big(\frac{D} {p^{n_p}-1} \Big)^{g_p}  =
\frac{1} {(p^{n_p}-1)^{g_p}} \sm_{D \in {\mathcal D}_p} D^{g_p}.  
\end{aligned}
\end{equation}

A trivial upper bound for $\ds \sm_{D \in {\mathcal D}_p} D^{g_p}$ is 
$\ds \sm_{k=1}^{p-1} k^{g_p} = O(1) \, p^{g_p+1}$, giving
$${\rm Prob}\big (o_p(\eta) < p\big ) \leq \frac{O(1)}{p^{g_p(n_p-1) -1}}$$ 
for which the application of  the Borel--Cantelli heuristic supposes the inequality 
$g_p(n_p-1) \geq 3$, giving possible obstructions for quadratic or cubic fields with $p$ inert, 
and quartic fields with $n_p =2$. 
Of course, if $g_p(n_p-1)$ increases, the heuristic becomes trivial and we can replace
${\rm Prob}\big (o_p(\eta) < p\big)$ by ${\rm Prob}\big (o_p(\eta) < p^\kappa \big)$, for some $\kappa>1$ 
(see Remark \ref{rema} (i)).
But we can remove the obstructions concerning the cubic and quartic cases using 
an analytic argument suggested by G. Tenenbaum:

\begin{theorem}\label{thmfond}
 Let $K/\Q$ be Galois of degree $n \geq 2$, of Galois group $G$, and let 
$\eta\in K^\times$ be such that the multiplicative $\Z[G]$-module generated by $\eta$ is of $\Z$-rank $n$.
For any prime number $p$, let $g_p$ be the number of prime ideals ${\mathfrak p} \mid p$ and let $n_p$ 
be the residue degree of $p$ in $K/\Q$. 

\smallskip\noindent
Then, under the above heuristic inequality \eqref{proba}, for all unramified $p \gg 0$ 
such that $n_p>1$, we have (where ${\rm log}_2 = {\rm log}\circ {\rm log}$)
$${\rm Prob} \big (o_p(\eta)<p \big) \leq \frac{1}{p^{g_p\,(n_p-1) - \varepsilon}}, \ \ 
\hbox{with $\,\varepsilon = O \ds\Big( \frac{1}{{\rm log}_2(p)}\Big)$ } . $$
\end{theorem}

\begin{proof}
Let $S_p := \sm_{D \in {\mathcal D}_p} D^{g_p}$; under the two conditions $D \mid p^{n_p}-1$, $D<p$, we get 
$S_p < \sm_{D \mid p^{n_p}-1} \Big(\frac{p}{D} \Big)^{g_p} D^{g_p} = p^{g_p} \cdot \tau(p^{n_p}-1)$,
where $\tau(m)$ denotes the number of divisors of the integer $m$. From \cite[Theorem I.5.4]{T},
we have, for all $c>{\rm log}(2)$ and for all $m \gg 0$, 
$$\tau(m) \leq m^{\frac{c}{{\rm log}_2(m)}}. $$
Taking $c=1$ and $m=p^{n_p}-1 < p^{n_p}$, this leads to $S_p < p^{g_p + \frac{n_p} {{\rm log}_2 (p^{n_p} - 1)}}$
for all $p \gg 0$. Thus

\medskip
$\hspace{1cm} \ds {\rm Prob}(o_p(\eta)<p) \leq \frac{S_p}{(p^{n_p}-1)^{g_p}} 
\leq \frac{1}{(p^{n_p}-1)^{g_p} \cdot p^{-g_p - \frac{n_p} { {\rm log}_2(p^{n_p}-1) }}}$

$\hspace{3.6cm}\ds = \frac{1}{p^{g_p\,(n_p-1) - O\big (\frac{1}{{\rm log}_2(p)} \big )}} $. 
\end{proof}

To apply the Borel--Cantelli heuristic giving the finiteness of primes $p$ such that $o_p(\eta) < p$, 
we must have the inequality $g_p\,(n_p-1) \geq \varepsilon+1$, hence $g_p\,(n_p-1)\geq 2$.
Otherwise, we get $g_p=1 \ \&\  n_p =2$, not sufficient to conclude for quadratic fields with $p$ inert 
since, in this case, 
$${\rm Prob}(o_p(\eta)<p) \leq \frac{1}{p^{1-\varepsilon}} \ \ 
\hbox{with $\varepsilon =  O\big(\frac{1}{{\rm log}_2(p)}\big)$.}$$

\begin{remarks}\label{rema}{\rm  (i) Still when $g_p\,(n_p-1) \geq 2$, we can replace the previous
inequality ${\rm Prob}\big (o_p(\eta) < p\big ) \leq \frac{1}{p^{g_p\,(n_p-1) - \varepsilon}}$ by 

\centerline{$\ds {\rm Prob}\big (o_p(\eta) < p^\kappa \big ) \leq \frac{1}{p^{g_p\,(n_p-1) - \varepsilon}}$}

\smallskip\noindent
which is true for any real $\kappa$ such that $1\leq \kappa \leq n_p - \frac{1 + \varepsilon}{g_p}$, 
in which case the Borel--Cantelli heuristic applies and may have some interest for large $n_p$; 
for instance, if $K=\Q_r$ is the subfield of degree $\ell^r$ ($\ell$ a prime, $r \geq 1$), of the cyclotomic 
$\Z_\ell$-extension of $\Q$, and if we take primes $p$ totally inert in $K/\Q$, one can take 
$\kappa = \ell^r-2$ (if $\ell^r \ne 2$) for any $\eta$ as usual.

\smallskip
(ii) Note that the proof of the theorem does not take into account the conditions 
$o_p(\eta) \nmid D_{n_p, \delta}(p)$ and it should be interesting to improve this aspect. 
But this theorem is a first step, and in the next sections, we intend to use explicitely the set 
${\mathcal D}_p$ for numerical computations and for a detailed study of the 
more ambiguous quadratic fields case.
 Indeed, in this case, we have to estimate the more precise 
upper bound $\frac{1}{p^2-1}\sum_{D \in {\mathcal D}_p} \phi(D)$
and a numerical experiment with the following PARI program (from \cite{PARI2})
shows a great dispersion of the number $N$ of such divisors:

\smallskip \footnotesize
\begin{verbatim}
 {b=10^5; B=b+10^3; forprime(p=b, B, N=0; my(e=kronecker(-4,p)); 
 F1=factor(2*(p-e)); F2=factor((p+e)/2); P=concat(F1[,1],F2[,1]); 
 E=concat(F1[,2],F2[,2]); forvec(v=vectorv(# E,i,[0,E[i]]), my(d=factorback(P,v)); 
 if(d>=p-1, next); if((p-1)%d!=0 && (p+1)%d!=0, N=N+1)); print(p," ", N))}
\end{verbatim}

\normalsize\noindent
giving for instance (depending on the factorizations of $p-1$ and $p+1$):

\medskip
$p=100237$,  $N=3$, where $(p-1) \cdot (p+1)= (2^2\cdot 3 \cdot 8353) \cdot (2 \cdot 50119)$,

\smallskip
$p=100673$, $N=489$, where $(p-1) \cdot (p+1)= 
(2^6\cdot 11^2 \cdot 13) \cdot (2 \cdot 3^2 \cdot 7 \cdot 17 \cdot 47)$.}
\end{remarks}

We shall return more precisely to the quadratic case in \S\,\ref{qcase}. 

\medskip
We can state to conclude this section:

\begin{conjecture}\label{mainconj}
Let $K/\Q$ be Galois of degree $n\geq 3$, of Galois group $G$. 
Let $\eta\in K^\times$ be such that the multiplicative $\Z[G]$-module 
generated by $\eta$ is of $\Z$-rank~$n$.  
For any unramified prime $p$, prime to $\eta$, let $o_p(\eta)$ 
be the order of $\eta$ modulo~$p$. 

\smallskip\noindent
Then $o_p(\eta)>p$, for all $p$ non totally split in $K/\Q$, except a finite number. 

\noindent
More generaly, $o_p(\eta)>p^{(n_p-1)+ 1-  \frac{n_p}{n} \cdot (1+\varepsilon)}$, 
for all $p$ such that $n_p > 1$, except a finite number.
\end{conjecture}

\section{Numerical evidences for the above conjecture}\label{sect7}

This section is independent of any $\eta \in K^\times$ and any number field $K$ but depends 
only on {\it given and fixed} integer parameters denoted by abuse $(n_p, \, g_p)$.
For $n_p >1$, we explicitely compute, for any $p$, $S_p := \sm_{D\in {\mathcal D}_p} D^{g_p}$
and the upper bound 
$\ds \frac{S_p}{(p^{n_p}-1)^{g_p}}$ of  $\ds {\rm Prob}\big (o_p(\eta) < p\big)$, using the program 
described below. Recall that 
$${\mathcal D}_p := \{D : \   D \mid p^{n_p}-1, \  D<p, \  D \nmid D_{n_p, \delta}(p) 
\  \forall \, \delta \mid n_p\}. $$

\subsection{General program about the divisors $D\in {\mathcal D}_p$}\label{program}
It is sufficient to precise the integers $n_p>1$, $g_p\geq 1$, and the interval $[b, B]$ of primes $p$.
The program gives the least value $C_b^B$ of $C(p)$, $p\in [b, B]$, where
$$\frac{S_p}{(p^{n_p}-1)^{g_p}} =: \frac{1}{p^{C(p)}}.$$
The favourable cases for the Borel--Cantelli principle are those with $C_b^B>1$, but the inequalities
$\ds C_b^B \geq C_b^\infty := {\rm Inf}_{p\in [b, \infty]} C(p)$
do not mean that the Borel--Cantelli principle applies since we ignore if $C_b^\infty>1$ 
or not for $b \gg 0$, because $C_b^\infty$ is an increasing function of $b$. 

\smallskip
In the applications given below, $n_p$ is a prime number, for which 
$D_{n_p, \delta}(p)$ is in $\{p-1, \, p^{n_p-1}+ \cdots + p+1\}$;
for more general values of $n_p$, one must first compute the set ${\mathcal D}_p$ as defined in 
Theorem \ref{thm}.

\smallskip\footnotesize
\begin{verbatim}
 {b=10^6; B=10^7; gp=1; np=2; CC= gp*(np-1)+1; C=CC; V=vector(B,i,i^gp);
 forprime(p=b, B, my(S=0, M=p^np-1, m=p-1, mm=M/m, i); 
 fordiv(M,d,if(d>p, break); if(m%d!=0 && mm%d!=0, S+=V[d])); 
 if(S!=0, C=(gp*log(M)-log(S))/log(p)); if(C<CC, CC=C)); print(CC)}
\end{verbatim}
 
\normalsize
\smallskip\noindent
The initial $CC:= g_p\,(n_p-1)+1 \geq 2$ is an obvious upper bound for $C_b^B$.

\subsection{Application to quadratic fields with $n_p=2$}  
We have $g_p=1$.
We obtain $C \approx 0.56402...$ for $10^6 \leq p \leq 10^7$, then $C \approx 0.58341...$ for 
$10^7 \leq p \leq 10^8$, and $C \approx 0.58326...$ for $10^8 \leq p \leq 10^9$.
For larger primes $p$ it seems that the constant $C$ stabilizes. 
If we replace $D$ by $\phi(D)$ the result is a bit better (e.g., $C\approx 0.64766...$ 
instead of $0.56402...$ for  $10^6 \leq p \leq 10^7$). 

\smallskip
The local extremum of $C$ are obtained
by primes $p$, like $166676399$, $604929599$, $1368987049$, $1758415231$, for which
$p^2-1$ is ``friable'' (product of small primes; see the computations that we shall 
give in \S\,\ref{qcase}).

\subsection{Application to cyclic cubic fields with $n_p=3$}
We use the program with $g_p=1$, $n_p=3$.
For instance, for $10^6 \leq p \leq 10^7$, we get $C\approx 1.5652... >1$ as expected
from Theorem \ref{thmfond}; for $10^7\leq p \leq  10^8$ the value of $C$ is $1.5399325...$
and for $10^8 \leq p \leq  3.65717251 \cdot10^8$, we get $C\approx 1.5809...$.

\subsection{Application to quartic fields with $n_p=2$ and for $n_p=4$}
For $g_p=2$, $n_p=2$, and $10^6 \leq p \leq 10^7$, we get $C=1.6103...$; for 
$10^8 \leq p \leq 10^9$, the result for $C$ is $1.6186...$. 

\smallskip
Naturally, for $n_p=4$ we obtain a larger constant $C=2.4596...$. But in the case $n_p=4$ we can test 
the similar stronger condition ${\rm Prob}\big (o_p(\eta) < p^2\big )$ for which one finds
$C = 1.28442...$, giving the conjectural finiteness of totally inert primes $p$ 
in a Galois quartic field such that $o_p(\eta) < p^2$.

\section{Numerical examples with fixed $\eta$ and $p\to\infty$}\label{sect8}

The above computations are of a density nature and the upper bound $\frac{1}{p^C}$ is much higher 
than the true probability. So we intend to take a fixed $\eta \in K^\times$, restrict ourselves to primes 
$p$ with suitable residue degree $n_p$, and compute the order of $\eta$ modulo $p$ to find 
the solutions $p$ of the inequality $o_p(\eta)<p$.

\smallskip
The programs verify that $\eta$ generates a multiplicative $\Z[G]$-module of  rank~$n$.
In the studied cases, $K/\Q$ is Abelian ($G= C_2$, $C_3$, $C_4$) and the condition on the rank is 
equivalent to $\eta^e\ne 1$ in $\langle \eta \rangle_G^{} \otimes \Q$, for all rational idempotents $e$ of $\Q[G]$.

\subsection{Cubic cyclic fields}
We then consider the following program with the polynomial $P=X^3+X^2-2\,X-1$ (see data in Section 3).
Put $\eta = a x^2+b x+c$; then $a$ is fixed and to expect more solutions, $b$, $c$ vary in 
$[-10, 10]$ and $p$ in $[3, 10^5]$:

\smallskip\footnotesize
\begin{verbatim}
 {P=x^3+x^2-2*x-1; x0=Mod(x,P); x1=-x0^2-x0+1; x2= x0^2-2;  
 Borne=10^5; a=1; for(b=-10, 10, for(c=-10, 10,   
 Eta0=a*x0^2+b*x0+c; Eta1=a*x1^2+b*x1+c; Eta2=a*x2^2+b*x2+c;  
 N=norm(Eta0); R1=Eta0*Eta1*Eta2; R2=Eta0^2*Eta1^-1*Eta2^-1;  
 if(R1!=1 & R2 !=1 & R1!=-1 & R2 !=-1, 
 forprime(p=1, Borne, if(p%N!=0, T=Mod(p,7)^2; if(T!=1, 
 A=Mod(a,p); B=Mod(b,p); C=Mod(c,p); Y=Mod(A*x^2+B*x+C, P); 
 my(m=p-1, mm=p^2+p+1); fordiv(m*mm, d, if(d>p, break);
 Z=Y^d; if(Z==1, print(a,"  ",b,"  ",c,"  ",p,"  ",d)))))))))}
\end{verbatim}

\normalsize
\smallskip\noindent
No solution is obtained except the following triples (the eventual multiples of $o_p(\eta)$ are not written):

\smallskip
$(a, b, c, p, o_p(\eta)) =$

\smallskip
$(1,  -7,  7,  {\bf 137, 56})$,  $(1,  -3,  3,  {\bf 37,  28})$, 
$(1,  4,  8,  {\bf 47,  37})$, $(1,  6,  -10,  {\bf 31, 18})$.

\smallskip
We have here an example ($\eta = x^2+4x+8$, $p=47$) where $o_p(\eta)=37$ divides 
$p^2+p+1 = 37\cdot 61$; 
this can be possible because $p$ is too small regarding $\eta^{s^2+s+1} = 1+8\,p = 377$ (see Lemma \ref{poly}).

\subsection{Quartic cyclic fields}
We consider the quartic cyclic field $K$ defined by the polynomial $P=X^4-X^3-6 X^2+X+1$ of discriminant $34^2$. 
The quadratic subfield of $K$ is $k = \Q(\sqrt {17})$ and $K = k\Big( \sqrt{(17+ \sqrt{17})/2} \ \Big)$. 
The program is analogous to the previous one with the parameters $n_p=g_p = 2$.
Put $\eta = a x^3+b x^2+c x +d$; then $b$, $c$, $d$ vary  in $[-10, 10]$, and $p$ in $[3, 10^5]$:

\medskip\footnotesize
\begin{verbatim}
 {P=x^4-x^3-6*x^2+x+1; x0=Mod(x,P); x1=-1/2*x0^3+3*x0+3/2;   
 x2=x0^3-x0^2-6*x0+1; x3=-1/2*x0^3+x0^2+2*x0-3/2;    
 Borne=10^5; a=1; for(b=-10, 10, for(c=-10, 10, for(d=-10, 10,   
 Eta0=a*x0^3+b*x0^2+c*x0+d; Eta1=a*x1^3+b*x1^2+c*x1+d;   
 Eta2=a*x2^3+b*x2^2+c*x2+d; Eta3=a*x3^3+b*x3^2+c*x3+d; N=norm(Eta0);   
 R1=Eta0*Eta1*Eta2*Eta3; R2=Eta0*Eta2^-1; R3=Eta0*Eta1^-1*Eta2*Eta3^-1;  
 if(R1!=1 & R2 !=1 & R3!=1 & R1!=-1 & R2 !=-1 & R3!=-1,   
 forprime(p=3, Borne, if(p%N!=0, if(issquare(Mod(p,17))==1, 
 u=sqrt(Mod(17, p)); v=(17+u)/2; if(issquare(v)==0,  
 A=Mod(a,p); B=Mod(b,p); C=Mod(c,p); D=Mod(d,p); Y=Mod(A*x^3+B*x^2+C*x+D,P); 
 my(m=p-1, mm=p+1); fordiv(m*mm, dd, if(dd>p, break); Z=Y^dd; 
 if(Z==1, print(a,"  ",b,"  ",c,"  ",d,"  ",p,"  ", dd)))))))))))}    
\end{verbatim}

\normalsize
\medskip
No solution is obtained except the following ones, where we consider 
at most a solution $(p, o_p(\eta))$ for a given $p$ (other solutions
may be given by conjugates of $\eta$ and/or by $\eta' \equiv \eta \pmod p$; 
many solutions with $p=19$ and the orders $12$ and $15$); we eliminate also the solutions
$(p, \lambda\,o_p(\eta))$, $\lambda >1$):

\medskip
$(a, b, c,d, p, o_p(\eta)) =$
$(1,  -10,  2,  -10,   {\bf 19,   12})$, $(1,  -10,  5,  -9,  {\bf 19, 15})$, 
 
$(1,  -9,  6,  9,  {\bf 43,   33})$, $(1,  -7,  -2,  -6,  {\bf 19, 8})$, $(1, -7,  2,  -8, {\bf 19, 10})$, 
 
$(1,  -8,  7,  7,  {\bf 461,   276})$,  $(1,  -4,  1,  8,  {\bf 1549,  1395})$, $(1,  -3,  0,  -6,  {\bf 223,   64})$,  

$(1,  -1,  -6,  -10,  {\bf 229,   184})$, 
$(1,  -1,  3,  -2,  {\bf 59,   40})$, $(1,  3,  -8,  6,  {\bf 53,   9})$, 

$(1,  3,  -5,  10,  {\bf 83,   21})$, $(1,  9,  -7,  5,  {\bf 43,   22})$.

\medskip
For the last three cases, the order divides $p+1$ for the same reason as above.
We have the more exceptional solution
$(1,  -4,  1,  8,  {\bf 1549,  1395})$ where $1395=9\cdot 5 \cdot 31$ 
with $9 \mid p-1$ and $5\cdot 31 \mid p+1$.

\subsection{Quadratic fields}
We consider the field $K$ defined by the polynomial $P=X^2-3$ and the following program with 
$\eta= a\,\sqrt 3 +b$, $a=1$, $b \in [-10, 0]$.

\medskip\footnotesize
\begin{verbatim}
 {m=3; P=x^2-m; x0=Mod(x,P); x1=-x0; a=1; Borne=10^5; 
 for(b=-10, 10, Eta0=a*x0+b; Eta1=a*x1+b; N=norm(Eta0); 
 R1=Eta0*Eta1; R2=Eta0*Eta1^-1; if(R1!=1 & R2 !=1 & R1!=-1 & R2 !=-1, 
 forprime(p=1, Borne, if(p%N!=0, 
 if(kronecker(m, p)==-1, A=Mod(a,p); B=Mod(b,p); Y=Mod(A*x+B,P);  
 my(m=p-1, mm=p+1); fordiv(m*mm, d, if(d>p, break);
 Z=Y^d; if(Z==1, print(a,"  ",b,"  ",p,"  ", d))))))))}
\end{verbatim} 

\normalsize
\medskip\noindent
For small primes $p$ there are solutions $o_p(\eta) \mid p-1$ or $o_p(\eta) \mid p+1$:

\medskip
$(a,\, b,\, p,\, o_p(\eta)) = $  

\smallskip
$(1,  -10, {\bf 79,  65})$, $(1,  -10 , {\bf 101, 75})$,  $(1,   -10, {\bf 967,   847})$, 

$(1,  -10 ,  {\bf 20359, 13234})$, $(1,  -10,    {\bf 90149,   72700})$, $(1,   -9,    {\bf 89,   55})$,  

$(1,  - 9,    {\bf 6163,   4623})$,  $(1,  -9,    {\bf 29501,   6705})$, $(1,  -8 ,   {\bf 10711,   2210})$, 

$(1,   -6,    {\bf 1123,   843})$,  $(1,  -5,    {\bf 86969,   81172})$, $(1,  -4,    {\bf 30941,   25785})$, 

$(1,  -9 ,   {\bf 41,   15})$, $(1,   -9,    {\bf 1301,   403})$, $(1,   -8,    {\bf 5,   3})$, $(1,   -7,    {\bf 29,   24})$, 

$(1,   -7,    {\bf 103,   39})$,  $(1,   -7,    {\bf 727,   143})$,   $(1,   -4,    {\bf 701,   675})$, $(1,   -3,    {\bf 43,   33})$.

\medskip
If Conjecture \ref{mainconj} is likely for degrees $n\geq 3$, the question arises for 
quadratic fields with $n_p=2$. 
We give here supplementary computations with the following simplified program which 
can be used changing $m, a, b$:

\medskip\footnotesize
\begin{verbatim}
 {m=3; a=5; b=2; Borne=10^9; forprime(p=1, Borne, if(kronecker(m, p)==-1, 
 A=Mod(a,p); B=Mod(b,p); P=x^2-m; Y=Mod(A*x+B, P);  my(e=kronecker(-4,p)); 
 F1=factor(2*(p-e)); F2=factor((p+e)/2); 
 P=concat(F1[,1],F2[,1]); E=concat(F1[,2],F2[,2]); 
 forvec(v=vectorv(# E,i,[0,E[i]]), my(d=factorback(P,v)); if(d>p, next); 
 Z=Y^d; if(Z==1, print(p,"  ",d)))))}
\end{verbatim}

\normalsize
\medskip
(i) For instance, if we fix $\eta=5\sqrt 3 +2$ and take larger primes inert in $\Q(\sqrt 3)$,
this gives the few solutions (up to $p\leq 10^9$):

\smallskip
$(p, o_p(\eta))=$ $(5, 4)$, $(29 ,  21)$, $(1063,  944)$, $(32707,   23384)$, $(90401,  68930)$. 

\medskip
(ii)  For $\eta=7\sqrt 3 +3$ we obtain the solutions  (up to $p\leq 10^9$):

\smallskip
$(p, o_p(\eta))=$ $(7, 6)$, $(29 ,  21)$, $(137,   92)$, $(7498769,   5927335)$, 

\hfill $(39208553, 31070928)$.

\noindent
The large solution $(p=39208553, o_p(\eta)=31070928)$ (where $p^2-1$ is friable)
is a bad indication for finiteness.

\medskip
(iii) Consider $K=\Q(\sqrt{-1})$ with $p\equiv 3 \pmod 4$ up to $p \leq 10^9$. 

\smallskip
For $\eta=\sqrt {-1} + 4$ (${\rm N}(\eta)= 17$), we obtain the solutions:

\smallskip
$(p, o_p(\eta)) = (49139, 19593)$, $(25646167,  22440397)$.

\smallskip
For $\eta=\sqrt {-1} + 2$ (${\rm N}(\eta)=5$), we obtain the solution:

\smallskip
$(p, o_p(\eta)) = (9384251, 6173850)$.  

\smallskip
For $\eta= 3\,\sqrt {-1} + 11$ (${\rm N}(\eta)=130$), we obtain the solutions:

\smallskip
$(p, o_p(\eta)) =$ $(3,   2)$, $(43,   11)$, $(131,   24)$, $(811,   174)$, $(911,   133)$, $(5743,   3168)$, 

\hfill $(2378711, 1486695)$.

\medskip
Although this kind of repartition of the solutions looks like the case of Fermat quotients, 
for which a specific heuristic can be used (see \cite{Gr2}), it seems that we observe more 
systematic large solutions in the quadratic case with $p$ inert, and we have possibly infinitely 
many solutions. This should be because the problem is of a different nature and is 
connected with generalizations of primitive roots problem in number fields 
(see the extensive survey by P. Moree \cite{Mo}). 

\smallskip
So we shall try in the next subsection to give some insights in the opposite direction for quadratic fields
(infiniteness of inert primes $p$ with $o_p(\eta) < p$).

\subsection{Analysis of the quadratic case}\label{qcase}
 Starting from the formula 
$${\rm Prob}\big (o_p(\eta) < p\big ) \leq {\rm Density}\big (o_p(\eta) < p\big ) =
\frac{1} {p^2-1} \sm_{D \in {\mathcal D}_p}\phi( D)$$

\noindent
of Remark \ref{rema}\ (iii), we study the right member of the normalized equality
$(p+1) \cdot {\rm Density}(o_p(\eta)<p) = \frac{1} {p-1} \sm_{D \in  {\mathcal D}_p}\phi( D)$,
remembering that it is an upperbound of the probability.
From numerical experiments, we can state:

\begin{conjecture}\label{conj}
{\it Let ${\mathcal D}_p$ be the set of divisors $D$ of $p^2-1$ such that $D<p$, 
$D\nmid p-1$, $D \nmid p+1$ (see Theorem \ref{thm}). We have the inequalities:
$$\frac{1}{3} \leq \frac{1} {p-1} \sm_{D \in  {\mathcal D}_p}\phi( D) 
< c(p) \, {\rm log}^2(p) , \ \ p\to\infty , $$

\noindent
where $c(p)$ is probably around $O({\rm log}_2(p))$. }
\end{conjecture}

\noindent
The majoration $\ds\frac{1} {p^2-1}\! \sm_{D \in  {\mathcal D}_p}  \phi( D) 
<  c(p) \cdot \frac{{\rm log}^2(p)}{p+1} \sim 
c(p) \cdot \frac{1}{p^{1 - 2 \cdot {\rm log}_2(p)/{\rm log}(p)}}$
is to be compared with the upper bound  $\frac{1}{p^{1-\varepsilon}}$ \big(with 
$\varepsilon = O\big(\frac{1}{{\rm log}_2(p)}\big)$\big) of Theorem \ref{thmfond}, 
but the sets of divisors $D \mid p^2-1$ are not the same and this information 
is only experimental. On the contrary, the minoration

\medskip
\centerline{$\ds\frac{1} {p-1} \sm_{D \in {\mathcal D}_p}\phi( D) \geq \frac{1}{3}$}

\smallskip\noindent
seems exact (except very few cases), and although the density ($\geq$ probability) is
$\frac{O(1)}{p}$, this suggests the possible infiniteness of inert $p$ such that $o_p(\eta) < p$ for 
fixed $\eta \in K^\times$ such that $\eta^{1+s}$ and $\eta^{1-s}$ are distinct from roots of unity.
Indeed, for $p \in \{2, 3, 5, 7, 17\}$, we get the strict inverse inequality 

\medskip
\centerline{$\ds\frac{1} {p-1} \sm_{D \in {\mathcal D}_p}\phi( D) < \frac{1}{3}$} 

\smallskip\noindent
and we have no other examples up to $10^9$. The equality 

\medskip
\centerline{$\ds\frac{1} {p-1} \sm_{D \in {\mathcal D}_p}\phi( D) = \frac{1}{3}$}

\smallskip\noindent
is doubtless {\it equivalent} to $p - 1 = 2^{u+2} \!\cdot 3^v \ \&\  p+1 = 2\cdot \ell$, 
for some $u \geq 0$, $v \geq 0$ and $\ell$ prime. 
To study this, one can use the following programs:

\smallskip
(i) Program testing the equality for any prime $p$.

\smallskip\footnotesize
\begin{verbatim}
 {b=1; B=10^9; forprime(p=b, B, my(S=0, e=kronecker(-4,p)); 
 F1=factor(2*(p-e)); F2=factor((p+e)/2); 
 P=concat(F1[,1],F2[,1]); E=concat(F1[,2],F2[,2]);  
 forvec(v=vectorv(# E,i,[0,E[i]]), my(d=factorback(P,v)); 
 if (d>p, next); if((p-1)%d!=0 && (p+1)%d!=0,
 S+= prod(i=1,# v, if(v[i],(P[i]-1)*P[i]^(v[i]-1),1)))); 
 if(3*S==p-1, print(p)))}
\end{verbatim}

\normalsize
\smallskip
(ii) Program giving the primes $p$ such that $p=1+2^{u+2}\! \cdot 3^v\ \&\ p=-1+ 2\cdot \ell$
(which are trivialy solutions).
We use the fact that it is easier to test in first the primality of $(p+1)/2$ for large $p$.

\smallskip\footnotesize
\begin{verbatim}
 {X=1; Y=1; T=1; J2=0; J3=0; K=0; L=listcreate(10^6); 
 while(T<10^1000, K=K+1; listput(L,T,K); 
 if(T==X, J2=J2+1; X=2*component(L,J2)); 
 if(T==Y, J3=J3+1; Y=3*component(L,J3)); 
 T=min(X,Y); p=1+T; if(isprime(p)==1, 
 my(S=0, e=kronecker(-4,p)); if(isprime((p+1)/2)==1, 
 F1=factor(2*(p-e)); F2=factor((p+e)/2); 
 P=concat(F1[,1],F2[,1]); E=concat(F1[,2],F2[,2]);
 forvec(v=vectorv(# E,i,[0,E[i]]), my(d=factorback(P,v));
 if (d>p, next); if((p-1)%d!=0 && (p+1)%d!=0,
 S+= prod(i=1,# v, if(v[i],(P[i]-1)*P[i]^(v[i]-1),1))));
 if(3*S==p-1, print(factor(p-1)," ",factor(p+1)," ",p)))))}
\end{verbatim}

\normalsize
\smallskip\noindent
We obtain the following solutions:

\smallskip\footnotesize
\begin{verbatim}
p-1             p+1
[2, 2; 3, 1]   [2, 1; 7, 1]    p=13
[2, 2; 3, 2]   [2, 1; 19, 1]    p=37
[2, 3; 3, 2]   [2, 1; 37, 1]    p=73
[2, 6; 3, 1]   [2, 1; 97, 1]    p=193
[2, 7; 3, 2]   [2, 1; 577, 1]    p=1153
[2, 5; 3, 4]   [2, 1; 1297, 1]    p=2593
[2, 2; 3, 6]   [2, 1; 1459, 1]    p=2917
[2, 11; 3, 6]  [2, 1; 746497, 1]    p=1492993
[2, 13; 3, 5]  [2, 1; 995329, 1]    p=1990657
[2, 16; 3, 4]  [2, 1; 2654209, 1]    p=5308417
[2, 20; 3, 3]  [2, 1; 14155777, 1]    p=28311553
[2, 20; 3, 8]  [2, 1; 3439853569, 1]    p=6879707137
[2, 28; 3, 8]  [2, 1; 880602513409, 1]    p=1761205026817
[2, 36; 3, 4]  [2, 1; 2783138807809, 1]    p=5566277615617
[2, 43; 3, 2]  [2, 1; 39582418599937, 1]    p=79164837199873
[2, 47; 3, 3]  [2, 1; 1899956092796929, 1]    p=3799912185593857
[2, 44; 3, 8]  [2, 1; 57711166318706689, 1]    p=115422332637413377
[2, 19; 3, 26] [2, 1; 666334875701477377, 1]    p=1332669751402954753
[2, 5; 3, 36]  [2, 1; 2401514164751985937, 1]    p=4803028329503971873
[2, 9; 3, 44]  [2, 1; 252101350959004475617537, 1]    p=504202701918008951235073
(......)                                                                                         
[2, 347; 3, 210] [2, 1; 2248236482316792976786964665292968461331995642040323695103
2046780867585152457721177889198712315934156013280843634240215226808653634390879379
03441584820738187206171506901838003018676481262351763229728833537, 1] 
p=44964729646335859535739293305859369226639912840806473902064093561735170304915442
3557783974246318683120265616872684804304536173072687817587580688316964147637441234
3013803676006037352962524703526459457667073
\end{verbatim}

\normalsize
It seems clear that the number of solutions may be infinite (with an exponential growth).

\medskip \noindent
Consider the following program: 

\smallskip\footnotesize
\begin{verbatim}
 {b=10^60+floor(Pi*10^35); forprime(p=b, b+10^3, my(S=0, e=kronecker(-4,p)); 
 F1=factor(2*(p-e)); F2=factor((p+e)/2);
 P=concat(F1[,1],F2[,1]); E=concat(F1[,2],F2[,2]);
 forvec(v=vectorv(# E,i,[0,E[i]]), my(d=factorback(P,v));
 if(d>p, next); if((p-1)%d!=0 && (p+1)%d!=0,
 S+= prod(i=1,# v, if(v[i], (P[i]-1)*P[i]^(v[i]-1),1))));
 Density=S/(p^2-1.0); Delta=S/(p-1.0)-1/3; C= Density*p/log(p);
 print(p,"  ", Density,"  ", Delta,"  ",C))}
\end{verbatim}

\normalsize
\smallskip
Then we obtain, for the inequalities  $\ds \frac{1}{3} \leq 
\frac{1} {p-1} \sm_{D \in {\mathcal D}_p}\phi( D) < c(p) \cdot {\rm log}^2(p)$, 
the following data, showing their great dispersion,
first for some small prime numbers, then for some larger ones, where

$\ds\ \bullet\ \  {\rm Density} := \frac{1} {p^2-1} \sm_{D \in  {\mathcal D}_p}\phi( D)$, 

$\ds\ \bullet\ \  \Delta :=(p+1) \cdot {\rm Density}  - \frac{1}{3} = 
\frac{1} {p-1} \sm_{D \in  {\mathcal D}_p}\phi( D) - \frac{1}{3}$, 

$\ds\ \bullet\ \  C := \frac{p}{{\rm log}(p)} \cdot {\rm Density} \ll c(p) \cdot {\rm log}(p)$:

\footnotesize
$$\begin{array}{cccrrccc}
{\rm prime\  number} \ \,  p   &&                            {\rm Density}        &  \Delta\ \  &    { C\ \ \ } \vspace{0.15cm}  \\ 
112771 &&                                                             1.35   \times 10^{-4}        &     14.9499  & 1.3137 \\ 
112787 &&                                                            3.43  \times 10^{-6}     &           0.0538     & 0.0332  \\
112799  &&                                                           1.03 \times 10^{-4}       &        11.2873  & 0.9989  \\
112807  &&                                                                 2.31 \times 10^{-5} &       2.2715  &    0.2239  \\
112831  &&                                                                     3.48  \times 10^{-5} &     3.5941  &     0.3376 \\
112843  &&                                                                    9.35  \times 10^{-6} &       0.7225  &    0.0907 \\
1000000012345678910111213141516172457  &&         3.39 \times 10^{-37}   &     0.0054   &    0.0040 \\
1000000012345678910111213141516172551   &&     1.13 \times 10^{-34}   &   112.7791  &    1.3645 \\
1000000012345678910111213141516172631  &&       2.02 \times 10^{-35}  &    19.9470   &    0.2446 \\
1000000012345678910111213141516172643  &&       9.88 \times 10^{-37}   &     0.6552    &    0.0119 \\
1000000012345678910111213141516172661  &&        1.69 \times 10^{-35}   &  16.5501  &    0.2036 \\
1000000012345678910111213141516172719  &&      6.83 \times 10^{-35}   &  67.9646   &   0.8239\\
10^{60} + 314159265358979323846264338327950343   &&   1.92 \times 10^{-58} &  192.1709  & 1.3934 \\
10^{60} + 314159265358979323846264338327950499   && 1.43 \times 10^{-59}  & 13.9993   &    0.1037 \\
10^{60} + 314159265358979323846264338327950541  &&  5.64 \times 10^{-59} &  56.0710  &    0.4082 \\
10^{60} + 314159265358979323846264338327950569  &&    7.50 \times 10^{-59}  & 74.6795  &   0.5429 \\
10^{60} + 314159265358979323846264338327950989  &&  2.63 \times 10^{-59} &  26.0318  &   0.1908 \\
10^{60} + 314159265358979323846264338327951201 &&    5.26 \times 10^{-59}  & 52.2864  &   0.3808   
\end{array} $$

\normalsize \medskip
(i) For $p=1000000012345678910111213141516172457$ above, we have:

\smallskip
$C \approx 0.004086$, $C/{\rm log}(p) \approx 4.930 \cdot 10^{-5}$,

$p-1= 2^3\cdot 3^2\cdot 389 \cdot 62528362319 \cdot 571006238831466292903$,

$p+1=2 \cdot 8131511 \cdot 61489187701134445376216864339$.

\smallskip
(ii) For $p=10123456789123456789125887$, we obtain $\Delta  \approx 5.0641\cdot 10^{-23}$,

\smallskip
$C \approx 0.005789$, $C/{\rm log}(p) \approx 10.054 \cdot 10^{-5}$,

$p-1= 2\cdot  5061728394561728394562943$,

$p+1=2^8 \cdot  3 \cdot 13181584360837834360841$.

\medskip
(iii) Large values of $C$ are, on the contrary, obtained when $p^2-1$
is the product of small primes (friable numbers). This may help to precise
the upper bound of $C$ since the local maxima increase slowly.
For instance: 

\medskip
$166676399^2-1 = 2^5 \cdot 3^3 \cdot 5^2 \cdot 7 \cdot 11 \cdot 17
\cdot 19 \cdot 23 \cdot 29 \cdot 31 \cdot 41 \cdot 61$
with $C \approx 41.91845$ and $C/{\rm log}(p) \approx 2.21421$. 

\medskip
$1758415231^2-1 = 2^8 \cdot 3^4 \cdot 5 \cdot 7 \cdot 11 \cdot 13 \cdot 17
\cdot 19 \cdot 29 \cdot 31 \cdot 37 \cdot 47 \cdot 59$
with $C \approx 81.51733$ and $C/{\rm log}(p) \approx 3.82932$. 

\medskip
The following program computes these successive local maxima:

\smallskip
\footnotesize
\begin{verbatim}
{B=10^20; CC=0.0; forprime(p=3, B, my(S=0, e=kronecker(-4,p)); F1=factor(2*(p-e)); 
F2=factor((p+e)/2); P=concat(F1[,1],F2[,1]); E=concat(F1[,2],F2[,2]);
forvec(v=vectorv(#E,i,[0,E[i]]), my(d=factorback(P,v)); if(d>p, next); 
if((p-1)%d!=0 && (p+1)%d!=0, S+= prod(i=1,#v,if(v[i],(P[i]-1)*P[i]^(v[i]-1),1))));
Pr=S/(p^2-1.0); C=Pr*p/log(p); if(C>CC, CC=C; print(p," ",CC," ",CC/log(p))))}

p             CC             CC/log(p)
11            0.1529118768   0.0637692056
19            0.2867929851   0.0974015719
29            0.3690965111   0.1096121427
(......)
604929599    51.9605419985   2.5696806133
1368987049   61.6784084466   2.9318543821
1758415231   81.5173320978   3.8293199014
\end{verbatim}
\normalsize
For $p>1758415231$ the running time becomes prohibitive 
although we may conjecture the infiniteness of these numbers.


\begin{thebibliography}{xx}

\bibitem{Gr1}  \textsc{G. Gras}, \textit{Les $\theta$-r\'egulateurs locaux d'un nombre alg\'ebrique :
Conjectures $p$-adiques}. Canadian Journal of Mathematics, Vol. {\bf 68}, 3 (2016), 571--624. 

\url{http://dx.doi.org/10.4153/CJM-2015-026-3}

\bibitem{Gr2} \textsc{G. Gras}, \textit{\'Etude probabiliste des quotients de Fermat}.
Functiones et Approximatio, Commentarii Mathematici, Vol. {\bf 54}, 1 (2016), 115--140. 

\url{http://projecteuclid.org/euclid.facm/1458656166#abstract}

\bibitem{Gr3} \textsc{G. Gras}, \textit{Class Field Theory: from theory to practice}. SMM, Springer-Verlag, 2003; 
second corrected printing, 2005.  \url{http://dx.doi.org/10.1007/978-3-662-11323-3}

\url{https://www.dropbox.com/s/thye2w2ialrafyh/rooteng%20-%20copie.pdf?dl=0}

\bibitem{Mo} \textsc{P. Moree}, \textit{ Artin's Primitive Root Conjecture - A Survey}. In: The John 
Selfridge Memorial Volume, Integers, Vol.  {\bf 12A}, 13 (2012), 1305--1416.

\url{http://www.integers-ejcnt.org/vol12a.html}
   
\bibitem{Nar} \textsc{W. Narkiewicz}, \textit{Elementary and Analytic Theory of Algebraic Numbers}.
Springer Monographs in Mathematics, 3rd Edition, Springer, 2004.

\url{http://dx.doi.org/10.1007/978-3-662-07001-7}

\bibitem{PARI2} The PARI~Group, PARI/GP version \texttt{2.9.0}, Universit\'e de Bordeaux, 2016.

\url{http://pari.math.u-bordeaux.fr/}.

\bibitem{T} \textsc{G. Tenenbaum}, \textit{Introduction \`a la Th\'eorie Analytique et Probabiliste des Nombres}.
$4^e$ \'edi\-tion revue et augment\'ee, Coll. \'Echelles, Belin, 2015.

\bibitem{Wa} \textsc{L.C. Washington}, \textit{Introduction to Cyclotomic Fields}. Graduate Texts in Math. 83, 
Springer enlarged second edition, 1997. \url{http://www.springer.com/us/book/9780387947624}


\end{thebibliography}
\end{document}